\documentclass{amsart}
\usepackage{latexsym,amssymb,enumerate, amsmath}
\usepackage{float}
\usepackage{color}

\newtheorem{theorem}{Theorem}[section]

\newtheorem{lemma}[theorem]{Lemma}

\theoremstyle{definition}
\newtheorem{definition}[theorem]{Definition}

\newtheoremstyle{named}{}{}{\itshape}{}{\bfseries}{.}{.5em}{\thmnote{#3 }#1}
\theoremstyle{named}

\definecolor{DOCTORi}{cmyk}{ 0.5, 0.9, 0, 0}

\begin{document}
	
\title[regular graph related to  the conjugacy class sizes]{On the regularity of a  graph related to conjugacy class sizes of a normal subgroup}
\author[Shabnam Rahimi]{ Shabnam Rahimi}
\address{Shabnam Rahimi, Faculty of Math. and Computer Sci., \newline Amirkabir University of Technology (Tehran Polytechnic), 15914 Tehran, Iran.}
\email{rahimi569@aut.ac.ir}

\maketitle

\begin{abstract}
Given a finite group $G$ with a normal subgroup $N$, the simple graph $\Gamma_\textit{G}( \textit{N} )$ is a graph whose vertices are of the form $|x^G|$, where $x\in{N\setminus{Z(G)}}$, and $x^G$ is the $G$-conjugacy class of $N$ containing the element $x$. Two vertices $|x^G|$ and $|y^G|$ are adjacent if they are not co-prime. In this article we prove that, if $\Gamma_G(N)$ is a connected incomplete regular graph, then  $N= P \times{A}$ where $P$ is a $p$-group, for some prime $p$ and $A\leq{Z(G)}$,   and  ${\bf Z}(N)\not = N\cap {\bf Z}(G)$. 
\end{abstract}
\color{black}

\section{Introduction}

Given a finite group $G$, by cs$(G)$ we mean the set of conjugacy class sizes of the group $G$. It is well known that strong results can be obtained from cs$(G)$ about the structure of $G$ (see \cite{camina2011influence} for example).

 Some certain graphs are introduced in order to study specific properties of a given finite group $G$. We are going to discuss the graphs which are constructed upon the set cs$(G)$.   
  The common divisor graph on conjugacy class sizes, that we denote  by $\Gamma(G)$ (see \cite{bertram1990graph}), is  a graph whose vertex set is cs$(G)\setminus\{1\}$ and vertices $v$ and $w$ are adjacent if gcd$(v,w)>1$. 
  In \cite{bianchi2012regularity} the graph $\Gamma(X)$ is used with somewhat similar definition, only it is defined on an arbitrary set of integers $X$.
  
  The properties of 
  this graph, regarding 
  its association to the algebraic structure of the group $G$, has been vastly investigated in the last few decades. We refer to \cite{lewis2008overview} for a survey on this topic.

  Let $N$ be a normal subgroup of the finite group $G$, the set cs$_G(N)$ denotes $G$-conjugacy class sizes of $N$. Discussing the structure of $N$ based  on cs$_G(N)$, could potentially extend the results that are made based on  cs$(G)$, and thus these properties have been studied actively in the recent years as well. It seems to be natural to define analogous graphs based on cs$_G(N)$. Denote by $\Gamma(${cs}$_G(N))$ or $\Gamma_G(N)$, the  graph  whose vertex set contains elements of the form $|x^G|$, where $x\in{N\setminus{Z(G)}}$ and vertices $v$ and $w$ are adjacent in $\Gamma_G(N)$ if and only if they are adjacent in $\Gamma(G)$ (see \cite{beltran2015graphs}).

In \cite{bianchi2012regularity} it is proved that  \textit{if $\Gamma(G)$ is $k$-regular then it must be a complete graph of order $k+1$} for $k=2,3$; and the result is extended for any $k\in \mathbb{N}$ in \cite{bianchi2015conjugacy}.

It seems reasonable to ask whether the same results hold when   $\Gamma_G(N)$ is a regular graph.  Note that, if $\Gamma_G(N)$ is a regular disconnected graph, then by \cite[Theorem B]{beltran2015graphs},  $\Gamma_G(N)$ has two complete components and the structure of $N$ in that case is determined. So the connected case is left to discuss.
In this paper, we aim to prove the following theorem as the main result: 

\begin{theorem}[Main]
 Let $G$ be a finite group and $N$ be a normal subgroup of $G$, such that $\Gamma_G(N)$ is a connected incomplete regular graph. Then  $N/(N\cap Z(G))$ is a $p$-group, for some prime $p$,   and  ${\bf Z}(N)\not = N\cap {\bf Z}(G)$.

\end{theorem}

\section{Preliminary}

\begin{definition}
For a given vertex $v$ of the graph $\Gamma$, define the neighborhood of $v$, the set of vertices adjacent to $v$, including $v$ itself and denote it by $\mathcal{N}_{\Gamma}(v)$.    
\end{definition}

\begin{definition} 
 Two distinct vertices $v_1, v_2$  of the graph $\Gamma$  are said to be  \textit{partners}, if:
  
\begin{center}
   $ \mathcal{N}_{\Gamma}(v_1)=\mathcal{N}_{\Gamma}(v_2)$
\end{center}
Observe that partnership provides an equivalence relation on the set of vertices of the graph.
\end{definition}

\begin{lemma} \label{elementpower}Let $|x^G|$  be a vertex of  $\Gamma_G(N)$, for some $x\in N$ and $\Gamma_G(N)$ be regular. If $y=x^a$ is non-central for some integer $a$, then either $|y^G|=|x^G|$ or $|x^G|$ and $|y^G|$ are partners.
\end{lemma}
\begin{proof}
 Assume $|x^G|\not = |y^G|$, then   $C_G(x)\subset C_G(y)$. Therefore $ |x^G|$ is divisible by $|y^G|$.  As  $|y^G|$ and $|x^G|$  have same degrees in $\Gamma_G(N)$, we conclude that they are partners. 
\end{proof}

\begin{lemma} \label{pfff} (see \cite{bertram1990graph,kazarin1981groups})
A finite group G satisfies $n(\Gamma(G))=2$ if and only if G is quasi-Frobenius and $G/{\bf Z}(G)$  has abelian kernel and complement.
\end{lemma}

\section{Main results}

\begin{lemma}\label{key}
		Let G be a finite group and $N$ be a normal subgroup of $G$ such that $|N/(N\cap Z(G))|$ is divisible by two distinct prime divisors $p_1$ and $p_2$.  Let  $x_0,y_0\in{N}$ be non-central $p_1$ and $p_2$-elements(respectively) such that $x_0y_0=y_0x_0$. Also, assume that $\Gamma_G(N)$ is a connected incomplete regular graph.  Then denoting by $v_0, w_0$ and $z_0$ the sizes of conjugacy classes of $G$, associated with $x_0, y_0$ and $x_0y_0$, respectively, the followings hold: 
	
		\begin{itemize}
		    \item[(a)]
			There exists a non-central $p_1$-element $x_1\in N$ and a non-central $p_2$-element  $y_1\in N$, such that  $v_1=|x_1^G|, w_1=|y_1^G|\in \mathcal{N}_{\Gamma_G(N)}(z_0)$, where  $v_1$ and $w_1$  are not adjacent in $\Gamma_G(N)$, $(v_1, p_1p_2)=p_2$ and $(w_1, p_1p_2)=p_1$.    
			
			\item[(b)]
			$v_0$ is divisible by $p_2$, $w_0$ is divisible by $p_1$. In particular $z_0$ is divisible by $p_1p_2$.
		\end{itemize}

\end{lemma}
\begin{proof}
	Since $\Gamma_G(N)$ is a connected incomplete regular graph, then for each vertex $v$ in $\Gamma_G(N)$, there exist two distinct  non-adjacent vertices in $\mathcal{N}_{\Gamma_G(N)}(v)$. Therefore non-central elements $x_1$ and $y_1$ exist in $N$ such that  $v_1=|x_1^G|$, $w_1=|y_1^G|$ and $(v_1,w_1)=1$, while they are both connected to $z_0=|{x_0y_0}^G|$. By Lemma \ref{elementpower}, we may assume that $o(x_1)$ and $o(y_1)$ are both power of some prime, distinct primes indeed.

Note that $p_1$ can not divide both $|x_1^G|$ and $|y_1^G|$, assuming that $p_1\nmid|x_1^G|$, thus $C_G(x_1)$ must contain some Sylow $p_1$-subgroup of $G$. Without loss of generality we may suppose that $x_1x_0=x_0x_1$. If $o(x_1)$ is not a power of $p_1$, by Lemma \ref{elementpower}, $v_1$ and $v_0$ should be partners, but that yields a contradiction, for then $v_1$ and $w_1$ must be adjacent. Therefore $x_1$ has to be a $p_1$-element. Accordingly by using same arguments $o(y_1)$ should be a power of $p_2$.

	Now, we prove $(b)$. 
	On the contrary, assume that $p_2\nmid{v_0}$ consequently,  $C_G(x_0)$ contains a  Sylow $p_2$-subgroup of $G$. Similar to the given arguments in the previous paragraph, we come to conclusion that $v_0$ and $w_1$ should be partners, and that contradicts the assumption stating the non-adjacency of $v_1$ and $w_1$. Similarly we get that $w_0$ is divisible by $p_1$. Therefore, since we know $v_0|z_0$ and $w_0|z_0$, $z_0$ is divisible by $p_1p_2$, as desired.

	\end{proof}

\begin{theorem}\label{2primes}
	Let G be a finite group and $N$ be a normal subgroup of $G$ such that $|N/(N\cap Z(G))|$ is divisible by two distinct primes $p_1$ and $p_2$. And $x_0,y_0\in{N}$ be non-central $p_1$ and $p_2$-elements(respectively) such that $x_0y_0=y_0x_0$. Also, assume $\Gamma_G(N)$ is a connected  regular graph. Then $\Gamma_G(N)$ is complete. 
	\end{theorem}
\begin{proof}
On the contrary assume that $\Gamma_G(N)$ is not complete, so Lemma \ref{key} would be applicable here.  Accordingly there exist such vertices $z_0:=|(x_0y_0)^G|$,  $v_0:=|x_0^G|$ and $w_0:=|y_0^G|$ as described in the statement of  Lemma \ref{key}, also let $v_1$ and $w_1$ be vertices such as those in the statement of Lemma \ref{key}(a); and define $A=\mathcal{N}_{\Gamma_{G}(N)}(v_1)\setminus \{z_0,v_1\}$.   

Assuming there exists an element $s\in N$ such that   $|s^G|\in A$ and  $(p_1p_2 , {|s^G|})=1$, then $s$ is an $r$-element, for some prime $r$. If 
 $r\neq{p_1}$, considering that $(p_1p_2 , {|s^G|})=1$ we may assume that $x_0s=sx_0$ then by applying Lemma \ref{elementpower}, the conclusion would be that vertices $v_0$ and ${|s^G|}$ must be partners, and so should $v_1$ and ${|s^G|}$, which implies that $v_1$ and $w_1$ are adjacent. So 
$s$  is a $p_1$-element. 

On the other hand, $C_G(s)$ contains a  Sylow $p_2$-subgroup of $G$ (it is mentioned above that $(p_1p_2,|s^G|)=1$), hence we may  assume that $s{y_1}={y_1}s$ and $sy_0=y_0s$.
Therefore, applying Lemma \ref{elementpower}, results in partnership of $|s^G|,{w_0},$ and ${w_1}$ which  implies adjacency of $v_1$ and $w_1$, a contradiction.
 Accordingly,    every vertex in   $ A$ is divisible by $p_1$ or $p_2$.

 Observe,  that $d(v_1)=|A|+1$. By regularity of $\Gamma_G(N)$ we have  $d(z_0)=|A|+1$. By Lemma \ref{key}(b), $z_0$ is adjacent to all vertices in $\{v_1,w_1\}\cup A$, which implies that $w_1\in A$, that again is a contradiction.   \\
\end{proof}

\begin{proof}[\textbf{Proof of the Main Theorem}]
Let $\pi(N/(N\cap {\bf Z}(G)))=\{p_1,\cdots, p_n\}$ where $p_i$s are distinct primes. By Theorem \ref{2primes},  for  every $p_i$-element $a_i$, $|a_i^G|$ is divisible  by  $({\prod_{j=1}^{n} p_j})/{p_i}$.  If $n>2$, we get that $\Gamma_G(N)$ is complete. Therefore we discuss the case where $n\leq 2$.

First, assume $n=2$. If ${\bf Z}(N)\nsubseteq {\bf Z}(G)$,  then  there exist $p_1$-element $x_0$ and $p_2$-element  $y_0$ in $N\setminus{Z(G)}$ such that $x_0y_0=y_0x_0$, and so by applying Theorem \ref{2primes} the graph must be complete. So let us discuss the case that ${\bf Z}(N)\subseteq {\bf Z}(G)$ and  $|{N}/{\textbf{Z}(N)}|={p_1^{n_1}}{p_2^{n_2}}$, for integers $n_1$ and $n_2$.

Note that we may assume for every non-central element $x\in N$, $|x^G|$ is not divisible by  $p_1p_2$, otherwise   $\Gamma_G( N )$ is complete. Therefore, for  every non-central $p_i$-element $x$, we have $|x^N|=p_j^{n_j}$, where $i\in \{1,2\}$ and $\{i,j\}=\{1,2\}$. We know that $\Gamma( N )$ is a disconnected graph with two vertices of sizes $p_i^{n_i}$, for $i=1,2$. Now, considering Lemma \ref{pfff}, $N$ is a quasi-Frobenius group with abelian kernel and complement. Therefore the Frobenius complement of $N/(N\cap {\bf Z}(G))$ is a cyclic $p_i$-group, for some $i=1,2$. Without loss of generality, we assume  that the Frobenius complement of $N/(N\cap {\bf Z}(G))$ is a cyclic $p_2$-group. Let $x$ be a non-central  $p_2$-element of $N$, such that $\langle x(N\cap {\bf Z}(G))\rangle$  is a Sylow $p_2$-subgroup of $N/(N\cap {\bf Z}(G))$.  Therefore, any non-central $p_2$-element of $N$ is $|x^G|$'s partner, in which case the graph must be complete. Consequently the case $n=1$ remains, as desired.\\
If ${\bf Z}(N)\subseteq {\bf Z}(G)$, the graph would be complete, since each vertex is divisible by $p$. Therefore, ${\bf Z}(N)\nsubseteq {\bf Z}(G)$ must be the case, and the proof is complete.
\end{proof}

\bibliographystyle{plain}
\bibliography{Bibliography.bib}

\end{document}